\documentclass[twoside,11pt]{article}
\usepackage{amsmath,amsthm,amssymb,color,algorithm,algorithmic,graphicx,float,subfigure,multirow,multicol}

\makeatletter
\newenvironment{breakablealgorithm}
  {
   \begin{center}
     \refstepcounter{algorithm}
     \hrule height.8pt depth0pt \kern2pt
     \renewcommand{\caption}[2][\relax]{
       {\raggedright\textbf{\ALG@name~\thealgorithm} ##2\par}%
       \ifx\relax##1\relax 
         \addcontentsline{loa}{algorithm}{\protect\numberline{\thealgorithm}##2}%
       \else 
         \addcontentsline{loa}{algorithm}{\protect\numberline{\thealgorithm}##1}%
       \fi
       \kern2pt\hrule\kern2pt
     }
  }{
     \kern2pt\hrule\relax
   \end{center}
  }
\makeatother

\graphicspath{{figures/}} 
\numberwithin{equation}{section}
\newtheorem{theorem}{\bf Theorem} 
\usepackage{jmlr2e}

\jmlrheading{1}{2000}{1-48}{4/00}{10/00}{Marina Meil\u{a} and Michael I. Jordan}

\firstpageno{1}

\begin{document}

\title{Composite Optimization Algorithms for Sigmoid Networks} 

\author{\name Huixiong Chen \email hxchen@m.scnu.edu.cn\\
       \addr
       School of Mathematical Sciences\\
       South China Normal University\\
       Guangzhou 510631, China\\
      \name Qi Ye \email yeqi@m.scnu.edu.cn \\
      \addr School of Mathematical Sciences\\
       South China Normal University\\
       Guangzhou 510631, China\\
       }


\maketitle

\begin{abstract}
In this paper, we use composite optimization algorithms to solve sigmoid networks. We equivalently transfer the sigmoid networks to a convex composite optimization and propose the composite optimization algorithms based on the linearized proximal algorithms and the alternating direction method of multipliers. Under the assumptions of the weak sharp minima and the regularity condition, the algorithm is guaranteed to converge to a globally optimal solution of the objective function even in the case of non-convex and non-smooth problems. Furthermore, the convergence results can be directly related to the amount of training data and provide a general guide for setting the size of sigmoid networks. Numerical experiments on Franke's function fitting and handwritten digit recognition show that the proposed algorithms perform satisfactorily and robustly.

\end{abstract}


\section{Introduction}

The neural network is an important and popular branch of machine learning. People have already developed many useful and well-studied neural network models, such as artificial neural networks, convolutional neural networks, recurrent neural networks, and deep neural networks. Neural networks have been widely used in pattern recognition, image processing, computer vision, neuroinformatics, bioinformatics, and other various fields with great success (\cite{LeCun2015, Abiodun2018}). 

When the neural networks are used in practical tasks, they are commonly trained by the error BackPropagation (BP) algorithm which is the most distinguished and successful neural network learning algorithm up to now. The BP algorithm is based on the gradient descent strategy that updates the parameters to the negative gradient direction of the target. To accelerate the learning process, stochastic gradient descent (SGD) with momentum and adaptive methods including adaptive gradient (AdaGrad), root mean square prop (RMSProp), adaptive moment estimation (Adam), and so on have emerged one after another and made a huge impact. As we all know, most of these first-order methods can converge to the critical point only if the objective function is convex or smooth. But for non-convex and non-smooth functions, it remains ambiguous how to find the convergence to even first- or second-order critical points (\cite{Burke2005}). Typical cases are sigmoid networks with absolute or hinge loss functions. The BP algorithm can solve these non-convex and non-smooth problems as well, but they are not consistent with the convergence properties of the algorithm. Moreover, it is still non-trivial to find globally optimal solutions for traditional neural network algorithms. We take the state-of-the-art Adam as an example. Its theory is poorly understood in the literature, and it suffers from several deficiencies. For instance, Adam may miss globally optimal solutions (\cite{Wilson2017}), and it can be shown that it does not converge on some simple test problems (\cite{Reddi2018}).

In this paper, we use composite optimization algorithms to solve sigmoid networks; see Algorithms \ref{lpa} and \ref{glpa} for details. The algorithm is guaranteed to (even globally) converge to a globally optimal solution of the objective function even in the case of non-convex and non-smooth problems. That is the main contribution of this paper. The start of our work stems from the finding that sigmoid networks (\ref{addition}) can be equivalently transformed into a convex composite optimization (\ref{composite}), where the inner function is smooth and the outer function is convex. This provides a new perspective on sigmoid networks. In fact, composite optimization problems arise in many applications in engineering, such as compressed sensing, image processing, machine learning, and artificial intelligence (\cite{Boyd2011, Hong2017}). The composite optimization is an area at the cutting edge of mathematical optimization, and how to efficiently solve composite optimization problems has been a popular subject.  For the sigmoid networks with the structure (\ref{composite}), the traditional first-order methods do not take advantage of the convex property of the outer function, so sometimes they have certain limitations in practical applications. However, composite optimization methods can fully exploit the information in the structure for algorithm design. There are many iterative algorithms with theoretical foundations for the optimization (\ref{composite}), such as the famous Gauss-Newton method (GNM, \cite{Burke1995}), the proximal descent algorithm (ProxDescent, \cite{Lewis2016}), and the linearized proximal algorithms (LPA, \cite{Hu2016}). The basic idea of these algorithms is to transfer a complex optimization problem to a sequence of simple optimization problems whose optimal solutions are easy to compute or have explicit formulas. The LPA is one of the most advanced algorithms in convex composite optimization. It can transform a non-convex and possibly non-smooth problem into a series of unconstrained strongly convex optimization subproblems, which has an attractive computational advantage. The LPA has also been applied to sensor network localization, gene regulatory network inference, and other engineering problems with great success (\cite{Hu2016, Hu2020, Wang2017}). Therefore, we use the LPA to solve sigmoid networks in this paper.

Under the assumptions of the weak sharp minima and the regularity condition, we establish the convergence behavior of the algorithms for sigmoid networks; see Theorems \ref{local} and \ref{global} for details. Furthermore, we prove that the weak sharp minima is often satisfied for sigmoid networks, and the full row rank of the Jacobian matrix of the inner function, namely $\text{\rm rank}(F^{\prime}(\bar{\boldsymbol{\theta}}))=m$, where $m$ is the amount of training data, is a sufficient condition of the regularity condition. Hence the convergence results can be directly related to the amount of training data; see Corollaries \ref{sufficient} and \ref{suffg} for details. This conclusion is of great theoretical and applied significance, especially since it can provide a general guide for setting the size of sigmoid networks. By the full row rank of the Jacobian matrix, we obtain a lower bound on the network size in Corollary \ref{net_size}. We call this lower bound the ``adaptive network size”. In this paper, our numerical experiments verify that the adaptive network size is sufficient to construct an ideal sigmoid network that solves the problem effectively. Hence Corollary \ref{net_size} does provide a good guide for setting the size of sigmoid networks. The essence is to guarantee that the number of parameters in neural networks is not smaller than the amount of training data, and that a sufficient number of parameters ensure the feasibility of the networks. It can also serve as a general guide for setting the size of neural networks. That is another contribution of this paper. 

Our work is also motivated by the lack of convex composite optimization algorithms and related software packages for neural networks. To the best of our knowledge, the introduction of convex composite optimization into the area of neural networks has not been addressed in the literature before. This paper is the first piece of work combining neural networks and convex composite optimization.


We organize the paper as follows. In section \ref{part2}, we introduce the three-layer sigmoid networks and transfer the problem to a convex composite optimization. In section \ref{part3}, we use the LPA-type algorithms to solve sigmoid networks and employ the alternating direction method of multipliers (ADMM) to solve the non-smooth convex subproblems. In section \ref{part4}, we prove some convergence properties of the proposed algorithms.
In section \ref{part5}, the numerical experiments are demonstrated including
Franke’s function fitting and handwritten digit recognition. Finally, we conclude with an outlook in section \ref{part6}.

\section{Sigmoid Networks}
\label{part2}
To begin with, we introduce the two-layer real-output sigmoid network, which is known as ‘universal approximators’ (\cite{Anthony1999}). Using the standard sigmoid function $\sigma: \mathbb{R}\rightarrow (0,1)$ of the form
\begin{equation*}
    \sigma(a) = \frac{1}{1+e^{-a}},
\end{equation*}
the sigmoid network computes a function $f:\mathbb{R}^d\rightarrow\mathbb{R}$ of the form
\begin{equation*}
    f(\boldsymbol{x})=\sum_{i=1}^{q} w_i\sigma(\boldsymbol{v}_i\cdot \boldsymbol{x}+u_i) + w_0,
\end{equation*}
where $w_i\in\mathbb{R}~(i=0,1,\dots,q)$ are the output weights,  $\boldsymbol{v}_i\in\mathbb{R}^d$ and $u_{i}\in\mathbb{R}~(i=1,2,\dots,q)$ are the input weights. We define these adjustable parameters by
\begin{equation*}
    \boldsymbol{\theta} = (\boldsymbol{w},\boldsymbol{v},\boldsymbol{u},w_0)^{T}\in\mathbb{R}^n,
\end{equation*} 
where $\boldsymbol{w}=(w_1,w_2,\dots,w_q), \boldsymbol{v} =(\boldsymbol{v}_1,\boldsymbol{v}_2,\dots,\boldsymbol{v}_q), \boldsymbol{u} =(u_1,u_2,\dots,u_q)$.
In the following paragraphs, we replace $f(\boldsymbol{x})$ with $f(\boldsymbol{x};\boldsymbol{\theta})$.
Given a training dataset 
\begin{equation*}
    D = \{ (\boldsymbol{x}_1,y_1),(\boldsymbol{x}_2,y_2),\dots,(\boldsymbol{x}_m,y_m)~|~\boldsymbol{x}_i\in\mathbb{R}^d, y_i\in\mathbb{R}\},
\end{equation*}
the goal of using this network for a supervised learning problem is to find parameters that minimize some measure of the error of the network output over the training dataset, that is,
\begin{equation}\label{addition}
    \min_{\boldsymbol{\theta}\in\mathbb{R}^n} ~E(\boldsymbol{\theta}):=\frac{1}{m} \sum_{i=1}^{m} L (f(\boldsymbol{x}_i;\boldsymbol{\theta}),y_i),
\end{equation}
where $L:\mathbb{R}\times\mathbb{R}\rightarrow[0,\infty)$ is a loss function. To simplify the discussions, we focus on three convex loss functions including the quadratic loss function, $(f(\boldsymbol{x};\boldsymbol{\theta}),y)\mapsto(f(\boldsymbol{x};\boldsymbol{\theta})-y)^2$, the absolute loss function, $(f(\boldsymbol{x};\boldsymbol{\theta}),y)\mapsto|f(\boldsymbol{x};\boldsymbol{\theta})-y|$, and the hinge loss function, $(f(\boldsymbol{x};\boldsymbol{\theta}),y)\mapsto(1-yf(\boldsymbol{x};\boldsymbol{\theta}))_{+}$.

The model of the sigmoid network is usually non-convex and non-smooth. Interestingly, we discover that this problem can be seen as a convex composite optimization problem of the form
\begin{equation}\label{composite}
    \min_{\boldsymbol{\theta}\in\mathbb{R}^n} ~E(\boldsymbol{\theta})=\mathbb{L}(F(\boldsymbol{\theta})),
\end{equation}
where the inner function $F:\mathbb{R}^n\rightarrow\mathbb{R}^m$ is smooth, and the outer function $\mathbb{L}:\mathbb{R}^{m}\rightarrow\mathbb{R}$ is convex. Specifically, for the absolute or quadratic loss functions, we can set
\begin{equation*} 
    F(\boldsymbol{\theta}) = \begin{pmatrix}
            f(\boldsymbol{x}_1;\boldsymbol{\theta})-y_1\\ f(\boldsymbol{x}_2;\boldsymbol{\theta})-y_2\\
            \vdots\\
            f(\boldsymbol{x}_m;\boldsymbol{\theta})-y_m
            \end{pmatrix}, ~ \mathbb{L}(\boldsymbol{z}) = \frac{1}{m}\Vert\boldsymbol{z}\Vert_p^p, ~\text{where}~ p=1 ~\text{or}~ 2.
\end{equation*}
In general, we replace $\Vert\cdot\Vert_2$ with $\Vert\cdot\Vert$. For the hinge loss function, we can set
\begin{equation*} 
    F(\boldsymbol{\theta}) = \begin{pmatrix}
            y_1 f(\boldsymbol{x}_1;\boldsymbol{\theta})\\
            y_2 f(\boldsymbol{x}_2;\boldsymbol{\theta})\\
            \vdots\\
            y_m f(\boldsymbol{x}_m;\boldsymbol{\theta})
            \end{pmatrix}, ~ \mathbb{L}(\boldsymbol{z}) = \frac{1}{m}\sum_{i=1}^{m}(1-z_i)_{+} = \frac{1}{m} \Vert(\boldsymbol{1}-\boldsymbol{z})_{+}\Vert_1, 
\end{equation*}
where $\boldsymbol{z}_{+}$ denotes the componentwise non-negative part of $\boldsymbol{z}$. As we can see, all the outer functions are separable and have the form
\begin{equation*}
    \mathbb{L}(\boldsymbol{z}) = \frac{1}{m}\sum_{i=1}^{m} \mathtt{L} (z_i),
\end{equation*}
where $\mathtt{L}:\mathbb{R}\rightarrow[0,\infty)$ is a convex function. It is the special property of $\mathbb{L}$ in sigmoid networks.


\section{Composite Optimization Algorithms for Sigmoid Networks}
\label{part3}

In this section, we show how to solve the sigmoid networks based on the composite optimization algorithms including the linearized proximal algorithms (LPA) and the alternating direction method of multipliers (ADMM).

{\bf 3.1 LPA for Sigmoid Networks}. The LPA is one of the most advanced algorithms in convex composite optimization. It is proposed under the inspiration of the GNM and the proximal point algorithm (PPA), and maintains the same convergence rate as that but also overcomes some of their disadvantages. Each subproblem of the LPA is constructed from a linearized approximation to the composite function and a regularization term at the current iterate. Since the subproblem is an unconstrained strongly convex optimization problem whose optimal solution is global and unique, it is easier to solve than that of the GNM. Consequently, the LPA has an attractive computational advantage, although it is generally not a descent algorithm.
 Moreover, there are some connections of the LPA with other algorithms mentioned in this paper. The ProxDescent for solving (\ref{composite}) is a special case of the LPA. As the descent directions are used, the ProxDescent is a descent algorithm. The case when the inner function is simply identity mapping has a long history. The iteration  $\boldsymbol{\theta}_{k+1}=\boldsymbol{\theta}_{k}+\Delta\boldsymbol{\theta}_k$, where $\Delta\boldsymbol{\theta}_k$ minimizes the function $\Delta\boldsymbol{\theta} \mapsto h(\boldsymbol{\theta}_k+\Delta\boldsymbol{\theta})+\frac{1}{2t}\Vert \Delta\boldsymbol{\theta}\Vert^2$, is the well-known PPA.

Applying the LPA directly to (\ref{composite}), we get the following algorithm for sigmoid networks.

\begin{breakablealgorithm}
    \caption{\hspace{-0.4em}. LPA for sigmoid networks.}
	\label{lpa}
	\renewcommand{\algorithmicrequire}{\textbf{Input:}}
	\renewcommand{\algorithmicensure}{\textbf{Output:}}
	\begin{algorithmic}[1]
	\REQUIRE Model $f$, training dataset $D=\{(\boldsymbol{x}_i,y_i)\}_{i=1}^{m}$, outer function $\mathbb{L}$, inner function $F$.
	\STATE \textbf{Initialization}: $t>0$, $\boldsymbol{\theta}_0\in\mathbb{R}^n$, $k\gets 0$, accept $\gets$ false;
	\WHILE{not accept}
	    \STATE calculate the search direction
		\begin{equation}\label{subp}
         \Delta\boldsymbol{\theta}_k := \mathop{\arg\min}\limits_{\Delta\boldsymbol{\theta}\in\mathbb{R}^n} \left\{ \mathbb{L} (F(\boldsymbol{\theta}_k)+F^{\prime}(\boldsymbol{\theta}_k)\Delta\boldsymbol{\theta})+\frac{1}{2t}\Vert \Delta\boldsymbol{\theta}\Vert^2 \right\},
        \end{equation}
        where $F^{\prime}(\boldsymbol{\theta}) = \left(\nabla^{T}_{\boldsymbol{\theta}} f(\boldsymbol{x}_i;\boldsymbol{\theta})\right)_{i=1}^m \in\mathbb{R}^{m\times n}$ is the Jacobian matrix of $F(\boldsymbol{\theta})$; 
        \IF{$\Delta\boldsymbol{\theta}_k=\boldsymbol{0}$}
            \STATE accept $\gets$ true;
        \ENDIF
         \STATE $\boldsymbol{\theta}_{k+1} \gets \boldsymbol{\theta}_k + \Delta\boldsymbol{\theta}_k$;
        \STATE $k\gets k+1$;
	\ENDWHILE
    \STATE $\boldsymbol{\theta}^{*}\gets \boldsymbol{\theta}_k$.
    \ENSURE $f(\boldsymbol{x};\boldsymbol{\theta}^{*})$.
	\end{algorithmic}
\end{breakablealgorithm}

The focus of Algorithm \ref{lpa} is how to solve the subproblem (\ref{subp}) accurately and efficiently. Now, we discuss some numerical algorithms for the special loss functions.

For the {\bf quadratic loss function}, Algorithm \ref{lpa} is reduced to the well-known Levenberg- Marquardt method for solving the following nonlinear least squares problem of the form
  \begin{equation*} 
      \min_{\boldsymbol{\theta}\in\mathbb{R}^n} ~E(\boldsymbol{\theta})=\frac{1}{m}\Vert F(\boldsymbol{\theta})\Vert^2.
  \end{equation*}
  The {\bf smooth convex subproblem} can be written as
  \begin{equation*} 
      \min_{\Delta\boldsymbol{\theta}\in\mathbb{R}^n} ~\frac{1}{m}\Vert F(\boldsymbol{\theta}_k)+F^{\prime}(\boldsymbol{\theta}_k)\Delta\boldsymbol{\theta}\Vert^2 + \frac{1}{2t}\Vert\Delta\boldsymbol{\theta}\Vert^2,
  \end{equation*}
   and its necessary and sufficient optimality conditions imply that
  \begin{equation*}
      \frac{2}{m} F^{\prime}(\boldsymbol{\theta}_k)^{T}\left(F(\boldsymbol{\theta}_k)+F^{\prime}(\boldsymbol{\theta}_k)\Delta\boldsymbol{\theta}_k\right) + \frac{1}{t}\Delta\boldsymbol{\theta}_k=\boldsymbol{0}.
  \end{equation*}
  Hence the closed formula of the search direction is given by
  \begin{equation*} 
      \Delta\boldsymbol{\theta}_k = - \left( \frac{2}{m}F^{\prime}(\boldsymbol{\theta}_k)^{T} F^{\prime}(\boldsymbol{\theta}_k) + \frac{1}{t}\boldsymbol{I} \right)^{-1}\left( \frac{2}{m}F^{\prime}(\boldsymbol{\theta}_k)^{T} F(\boldsymbol{\theta}_k) \right) \triangleq -B_k^{-1} \nabla E(\boldsymbol{\theta}_k),
  \end{equation*}
  where $B_k$ is always a positive-definite and invertible matrix, and $\nabla E(\boldsymbol{\theta})$ is the gradient of the objective function in the original problem. Thus, the iteration $\boldsymbol{\theta}_{k+1} = \boldsymbol{\theta}_k-B_k^{-1} \nabla E(\boldsymbol{\theta}_k)$ can be regarded as a variant of gradient descent algorithm, where $B_k^{-1}$ is an adaptive learning rate. Moreover, the  stopping criterion $\Delta\boldsymbol{\theta}_k=\boldsymbol{0}$ shows that $\nabla E(\boldsymbol{\theta}_k)=\boldsymbol{0}$, which is the first-order necessary condition of the original problem. In section \ref{part4}, we will give a first-order sufficient condition of the original problem in Theorem \ref{suffr}. 

  {\bf 3.2 ADMM for Non-Smooth Convex Subproblems}. For the non-smooth convex loss functions, the subproblem of Algorithm \ref{lpa} is more complex, but luckily it is convex. There are many widely used convex optimization methods and heuristic algorithms to solve it, such as gradient or subgradient methods, approximation or composite optimization methods (\cite{Bertsekas2015}), and simulated annealing algorithms. Moreover, there are many related software packages to implement these algorithms, such as CVXPY in Python, qpOASES in C$++$, and CVX toolbox in Matlab. So it is not difficult to calculate the search direction from the subproblem. We use a mapping A to represent a specific algorithm to solve the subproblem, then the search direction can be presented in
\begin{equation*}
    \Delta\boldsymbol{\theta}_k = {\rm A}(\mathbb{L}, F, F^{\prime}, t, \boldsymbol{\theta}_k).
\end{equation*}
Here we use the ADMM to solve the subproblem. The ADMM is a simple scheme that often works well and has a good reliability with a wide range of applications, especially for convex problems. It is also easy to understand and implement for many composite optimization problems with complex structures (\cite{Boyd2011}).

The subproblem (\ref{subp}) can be seen as an equivalent problem of the form
\begin{align*}
      \min_{\boldsymbol{\mu}\in\mathbb{R}^{m}, \Delta\boldsymbol{\theta}\in\mathbb{R}^n} ~&~ \mathbb{L}(\boldsymbol{\mu}) + \frac{1}{2t}\Vert\Delta\boldsymbol{\theta}\Vert^2,\\
      \text{s.t.}\hspace{1.3em} ~&~ \boldsymbol{\mu} - F(\boldsymbol{\theta}_k)-F^{\prime}(\boldsymbol{\theta}_k)\Delta\boldsymbol{\theta} = \boldsymbol{0}.\notag
  \end{align*}
  The augmented Lagrangian function of the above problem is
  \begin{align*}
     \mathcal{L}_{\rho}(\boldsymbol{\mu},\Delta\boldsymbol{\theta},\boldsymbol{\lambda}) = \mathbb{L}(\boldsymbol{\mu}) & + \frac{1}{2t}\Vert\Delta\boldsymbol{\theta}\Vert^2 + \boldsymbol{\lambda}^T \left( \boldsymbol{\mu} - F(\boldsymbol{\theta}_k)-F^{\prime}(\boldsymbol{\theta}_k)\Delta\boldsymbol{\theta} \right) + \frac{\rho}{2} \Vert \boldsymbol{\mu} - F(\boldsymbol{\theta}_k)-F^{\prime}(\boldsymbol{\theta}_k)\Delta\boldsymbol{\theta} \Vert^2,
  \end{align*}
  where $\rho>0$ is the penalty parameter. The ADMM consists of the iterations
  \begin{align}
      \boldsymbol{\mu}^{i+1} & := \mathop{\arg\min}\limits_{\boldsymbol{\mu}\in\mathbb{R}^{m}} \mathcal{L}_{\rho}(\boldsymbol{\mu},\Delta\boldsymbol{\theta}^i,\boldsymbol{\lambda}^i),\notag 
      \\
      \Delta\boldsymbol{\theta}^{i+1} & := \mathop{\arg\min}\limits_{\Delta\boldsymbol{\theta}\in\mathbb{R}^n} \mathcal{L}_{\rho}(\boldsymbol{\mu}^{i+1},\Delta\boldsymbol{\theta},\boldsymbol{\lambda}^i),\notag \\
      \boldsymbol{\lambda}^{i+1} & := \boldsymbol{\lambda}^{i} + \rho\left(\boldsymbol{\mu}^{i+1} - F(\boldsymbol{\theta}_k) - F^{\prime}(\boldsymbol{\theta}_k)\Delta\boldsymbol{\theta}^{i+1} \right). \label{eta1} 
  \end{align}

   The calculation for $\boldsymbol{\mu}^{i+1}$ is as follows.
  \begin{align} \label{z1}
      \boldsymbol{\mu}^{i+1} & =  \mathop{\arg\min}\limits_{\boldsymbol{\mu}\in\mathbb{R}^{m}}~ \left\{\mathbb{L}(\boldsymbol{\mu}) + \frac{\rho}{2} \Vert \boldsymbol{\mu} - F(\boldsymbol{\theta}_k) - F^{\prime}(\boldsymbol{\theta}_k)\Delta\boldsymbol{\theta}^i + \frac{1}{\rho}\boldsymbol{\lambda}^i \Vert^2\right\},\notag \\
      & = \mathop{\arg\min}\limits_{\boldsymbol{\mu}\in\mathbb{R}^{m}}~ \left\{\mathbb{L}(\boldsymbol{\mu}) + \frac{\rho}{2} \Vert \boldsymbol{\mu} - \boldsymbol{a}^i \Vert^2\right\},\notag \\
      & = \mathop{\arg\min}\limits_{\boldsymbol{\mu}\in\mathbb{R}^{m}}~ \left\{\sum_{j=1}^{m} \left( \frac{1}{m} \mathtt{L}(\mu_j) + \frac{\rho}{2} (\mu_j - a^i_j)^2\right)\right\},\notag \\
      & = \left( \mathop{\arg\min}\limits_{\mu_j\in\mathbb{R}}~\left\{ \frac{1}{m}\mathtt{L}(\mu_j) + \frac{\rho}{2} (\mu_j - a^i_j)^2\right\} \right)_{j=1}^m,\notag \\
      & = \left( \Phi_{1/{m\rho}}(a^i_j) \right)_{j=1}^m,
  \end{align}
  where $\boldsymbol{a}^i = F(\boldsymbol{\theta}_k) + F^{\prime}(\boldsymbol{\theta}_k)\Delta\boldsymbol{\theta}^i - \frac{1}{\rho}\boldsymbol{\lambda}^i$, and $\Phi_{\kappa}$ is the proximity operator of $\mathtt{L}$ with the penalty $\frac{1}{\kappa}$ (\cite{Boyd2011}). Specifically, for the {\bf absolute loss function}, the proximity operator $\Phi$, also called the soft thresholding operator, is defined as
  \begin{equation*}
      \Phi_{\kappa}(a) = \left\{ \begin{array}{cl}
       a-\kappa, & a>\kappa, \\[7pt]
       0,  & |a|\leq\kappa, \\[7pt]
       a+\kappa, & a<-\kappa.
    \end{array} \right.
  \end{equation*}
  For the {\bf hinge loss function}, the proximity operator $\Phi$ is defined as
\begin{equation*}
      \Phi_{\kappa}(a) = \left\{ \begin{array}{cc}
       a, & a>1, \\[7pt]
       1,  & 1-\kappa\leq a\leq 1, \\[7pt]
       a+\kappa, & a<1-\kappa.
    \end{array} \right.
  \end{equation*}

  The calculation for $\Delta\boldsymbol{\theta}^{i+1}$ is as follows. Since
  \begin{equation*}
      \Delta\boldsymbol{\theta}^{i+1} = \mathop{\arg\min}\limits_{\Delta\boldsymbol{\theta}\in\mathbb{R}^n}
      ~\left\{\frac{1}{2t}\Vert\Delta\boldsymbol{\theta}\Vert^2 + \frac{\rho}{2} \Vert \boldsymbol{\mu}^{i+1} - F(\boldsymbol{\theta}_k) - F^{\prime}(\boldsymbol{\theta}_k)\Delta\boldsymbol{\theta} + \frac{1}{\rho}\boldsymbol{\lambda}^i \Vert^2\right\},
  \end{equation*}
  by its necessary and sufficient optimality conditions, we obtain that
  \begin{equation} \label{d1}
      \Delta\boldsymbol{\theta}^{i+1} = \left( \rho F^{\prime}(\boldsymbol{\theta}_k)^{T} F^{\prime}(\boldsymbol{\theta}_k) +\frac{1}{t}\boldsymbol{I} \right)^{-1} \left( \rho F^{\prime}(\boldsymbol{\theta}_k)^{T} \left( \boldsymbol{\mu}^{i+1} - F(\boldsymbol{\theta}_k) + \frac{1}{\rho}\boldsymbol{\lambda}^i \right)\right).
    \end{equation}

    As we can see, the iterations of the ADMM for the non-smooth convex subproblems have explicit formulas, which is one of the advantages of the ADMM.
    Defining the primal residual of the optimality conditions at iteration $i$ as
  \begin{equation} \label{pr}
      \boldsymbol{r}^{i} = \boldsymbol{\mu}^{i} - F(\boldsymbol{\theta}_k) - F^{\prime}(\boldsymbol{\theta}_k)\Delta\boldsymbol{\theta}^{i},
  \end{equation}
  and the dual residual at iteration $i$ as
  \begin{equation} \label{dr}
      \boldsymbol{s}^{i} = \rho F^{\prime}(\boldsymbol{\theta}_k)(\Delta\boldsymbol{\theta}^{i}-\Delta\boldsymbol{\theta}^{i-1}),
  \end{equation}
 we set the stopping criterion as $\Vert\boldsymbol{r}^{i}\Vert \approx 0$ and $\Vert\boldsymbol{s}^{i}\Vert \approx 0$. 

    \renewcommand{\thealgorithm}{A$^*$}
    \begin{breakablealgorithm}
    \caption{\hspace{-0.3em}. ADMM for non-smooth convex subproblems.}
    \label{admm}
	\renewcommand{\algorithmicrequire}{\textbf{Input:}}
	\renewcommand{\algorithmicensure}{\textbf{Output:}}
	\begin{algorithmic}[1]
	\REQUIRE Numbers $m$ and $t$, matrices $F(\boldsymbol{\theta}_k)$ and $F^{\prime}(\boldsymbol{\theta}_k)$, non-smooth convex function $\mathtt{L}$.
	\STATE \textbf{Initialization}: $\rho>0$, $\Delta\boldsymbol{\theta}^0\in\mathbb{R}^n$, $\boldsymbol{\lambda}^0\in\mathbb{R}^m$,  $\epsilon>0$, $i\gets 0$;
	\REPEAT
	\STATE $i\gets i+1$;
	\STATE calculate \hspace{0.1em}$\boldsymbol{\mu}^i$\hspace{0.1em} from (\ref{z1}); 
	\STATE calculate \hspace{0.1em}$\Delta\boldsymbol{\theta}^i$\hspace{0.1em} from (\ref{d1}); 
	\STATE calculate \hspace{0.1em}$\boldsymbol{\lambda}^i$\hspace{0.1em} from (\ref{eta1}); 
	\STATE calculate \hspace{0.05em}$\boldsymbol{r}^i$\hspace{0.05em} and \hspace{0.05em}$\boldsymbol{s}^i$\hspace{0.05em} from (\ref{pr}) and (\ref{dr}), respectively;
	\UNTIL $\Vert\boldsymbol{r}^i\Vert<\epsilon$ and $\Vert\boldsymbol{s}^i\Vert<\epsilon$;
	\STATE $\Delta\boldsymbol{\theta}_k\gets \Delta\boldsymbol{\theta}^{i}$.
    \ENSURE $\Delta\boldsymbol{\theta}_k$.
	\end{algorithmic}
\end{breakablealgorithm}

{\bf 3.3 A Globalization Strategy for Algorithm \ref{lpa}}. Moreover, we show the following algorithm by employing the globalized LPA (GLPA) that adopts a backtracking line-search as a globalization strategy. The choice of the stepsize is based on the virtue of the backtracking line-search, which guarantees the monotone decrease of the objective function at each iteration. As a result, it ensures that the GLPA is a descent algorithm. In the algorithm implementation, the backtracking strategy finds the first point satisfying the inequality (\ref{rule}) by continuously decreasing the trial stepsize in an exponential way. That makes the stepsize with the descent property as large as possible.

    \renewcommand{\thealgorithm}{2}
    \begin{breakablealgorithm}
    \caption{\hspace{-0.3em}. GLPA for sigmoid networks.}
	\label{glpa}
	\renewcommand{\algorithmicrequire}{\textbf{Input:}}
	\renewcommand{\algorithmicensure}{\textbf{Output:}}
	\begin{algorithmic}[1]
	\REQUIRE Model $f$, training dataset $D=\{(\boldsymbol{x}_i,y_i)\}_{i=1}^{m}$, outer function $\mathbb{L}$, inner function $F$.
	\STATE \textbf{Initialization}: $t>0$, $c,\tau\in(0,1)$, $\boldsymbol{\theta}_0\in\mathbb{R}^n$, $k\gets 0$, accept $\gets$ false;
	\WHILE{not accept}
	    \STATE calculate the search direction
		\begin{equation*}
         \Delta\boldsymbol{\theta}_k = \mathop{\arg\min}\limits_{\Delta\boldsymbol{\theta}\in\mathbb{R}^n} \left\{ \mathbb{L} (F(\boldsymbol{\theta}_k)+F^{\prime}(\boldsymbol{\theta}_k)\Delta\boldsymbol{\theta})+\frac{1}{2t}\Vert \Delta\boldsymbol{\theta}\Vert^2 \right\};
        \end{equation*}
        \IF{$\Delta\boldsymbol{\theta}_k=\boldsymbol{0}$}
            \STATE accept $\gets$ true;
        \ENDIF
        \STATE $\eta\gets 1/\tau$;
        \REPEAT
		\STATE $\eta\gets\tau\eta$
		\UNTIL
		 \begin{equation} \label{rule}
         \mathbb{L}(F(\boldsymbol{\theta}_k+\eta\Delta\boldsymbol{\theta}_k))- \mathbb{L}(F(\boldsymbol{\theta}_k)) \leq c \hspace{0.1em}\eta\left(
         \mathbb{L}(F(\boldsymbol{\theta}_k)+F^{\prime}(\boldsymbol{\theta}_k)\Delta\boldsymbol{\theta}_k) + \frac{1}{2t}\Vert \Delta\boldsymbol{\theta}_k\Vert^2 - \mathbb{L}(F(\boldsymbol{\theta}_k)) \right);
        \end{equation}\ \\[-18pt]
        \STATE $\eta_k\gets\eta$;
         \STATE $\boldsymbol{\theta}_{k+1} \gets \boldsymbol{\theta}_k + \eta_k\Delta\boldsymbol{\theta}_k$;
        \STATE $k\gets k+1$;
	\ENDWHILE
    \STATE $\boldsymbol{\theta}^{*}\gets \boldsymbol{\theta}_k$.
    \ENSURE $f(\boldsymbol{x};\boldsymbol{\theta}^{*})$.
	\end{algorithmic}
\end{breakablealgorithm}

\section{Convergence Analysis}
\label{part4}

In this section, we prove some convergence properties of the proposed algorithms under the assumptions of the weak sharp minima and the regularity condition or full row rank of the Jacobian matrix, a stronger condition. Before giving the main results, we introduce the following useful definitions and lemmas.\\[-10pt]


{\bf 4.1 Theoretical Foundations of LPA-type Algorithms.} Here we consider the convex composite optimization of the form
\begin{equation} \label{cco}
    \min_{\boldsymbol{\omega}\in\mathbb{R}^b} ~h(G(\boldsymbol{\omega})),
\end{equation}
where the inner function $G:\mathbb{R}^b\rightarrow\mathbb{R}^l$ is continuously differentiable, and the outer function $h:\mathbb{R}^{l}\rightarrow\mathbb{R}$ is convex. It is a more general mathematical form of the problem (\ref{composite}).

 First, we introduce the concept of the Lipschitz continuous gradient, which has played an important role in investigating the convergence behavior of many optimization algorithms.
For a differentiable function $G$ and $\Omega \subseteq \mathbb{R}^b$, if there exists an $K>0$ such that
\begin{equation*}
    \Vert G^{\prime}(\tilde{\boldsymbol{\omega}}_1)-G^{\prime}(\tilde{\boldsymbol{\omega}}_2)\Vert \leq K \Vert \tilde{\boldsymbol{\omega}}_1-\tilde{\boldsymbol{\omega}}_2\Vert ~\text{for each}~ \tilde{\boldsymbol{\omega}}_1,\tilde{\boldsymbol{\omega}}_2\in\Omega,
\end{equation*}
we say that $G$ is \emph{K-smooth} or has a \emph{Lipschitz continuous gradient} with modulus $K$ on $\Omega$.

Next, we give the notion of the weak sharp minima introduced in (\cite{Burke1993}), which has far-reaching consequences for the convergence analysis of many iterative procedures. For a function $h$, the minimum value and the set of minima for $h$, denoted by $h_\text{min}$ and $C_{h}$, are defined by
    \begin{equation*}
       h_{\text{min}} := \min_{\boldsymbol{z}\in\mathbb{R}^l} h(\boldsymbol{z})~~\text{and}~~ C_{h}:= \mathop{\arg\min}_{\boldsymbol{z}\in\mathbb{R}^l}\hspace{0.15em} h(\boldsymbol{z}).
    \end{equation*}
    Let $C_{h} \subseteq S \subseteq \mathbb{R}^l$, if there exist $\alpha>0$ and $\beta\geq 1$ such that
    \begin{equation*}
       h(\boldsymbol{z})\geq h_{\text{min}}+\alpha\hspace{0.02em}\text{dist}^{\hspace{0.02em}\beta} (\boldsymbol{z},C_{h}) ~\text{for each}~\boldsymbol{z}\in S,
    \end{equation*}
    where $\text{dist} (\boldsymbol{z},C) := \inf_{\boldsymbol{c}\in C}\Vert \boldsymbol{z}-\boldsymbol{c}\Vert$, then we say that $C_{h}$ is the set of \emph{weak sharp minima} of order $\beta$ for $h$ on $S$ with modulus $\alpha$.

 We now introduce the regularity condition proposed in (\cite{Burke1995}), which is a crucial assumption applied to establish the convergence of several convex composite optimization algorithms. Let $h$ and $G$ be defined by {(\ref{cco})}, then a point $\bar{\boldsymbol{\omega}}\in\mathbb{R}^b$ is said to be a \emph{regular point} of the inclusion $G(\boldsymbol{\omega})\in C_{h}$ if
    \begin{equation*}
        \text{\rm ker}(G^{\prime}(\bar{\boldsymbol{\omega}})^T) \cap (C_{h}-G(\bar{\boldsymbol{\omega}}))^{\ominus} = \{\boldsymbol{0}\},
    \end{equation*}
    where ker($W$) $:=\{\boldsymbol{y}: W\boldsymbol{y}=\boldsymbol{0}\}$ is the nullspace of $W$, and $Z^{\ominus} := \{ \boldsymbol{y}:\left<\boldsymbol{y},\boldsymbol{z}\right>\leq 0, \forall \boldsymbol{z}\in Z\}$ is the negative polar of $Z$.

  In the following lemmas, we give the local convergence of the LPA and the global convergence of the GLPA for solving optimization (\ref{cco}). They are based on three main conditions including Lipschitz continuous gradient, weak sharp minima and quasi-regularity or regularity condition. Note that the definition of quasi-regularity condition will only be described in the proof of Theorem \ref{local}. Since this condition is hard to verify in practice, we replace it with the regularity condition in the related theorem for sigmoid networks.

\begin{lemma} \label{lem_lpa} {\bf\rm(\cite{Hu2016}, Corollary 14)} Let \hspace{0.15em}$\bar{\boldsymbol{\omega}}\in\mathbb{R}^b$ satisfy $G(\bar{\boldsymbol{\omega}})\in C_{h}$, and let $C_{h}$ be the set of weak sharp minima of order $\beta$ for $h$ near $G(\bar{\boldsymbol{\omega}})$ with constant $\alpha$. Suppose that $G$ is continuously differentiable with a Lipschitz continuous gradient $G^{\prime}$ near $\bar{\boldsymbol{\omega}}$, and that $\bar{\boldsymbol{\omega}}$ is a quasi-regular point of the inclusion with constant $\delta$. Suppose further that $\beta\in[1,2)$ or the stepsize $t>\frac{2\delta^2}{\alpha}$ {\rm(}if $\beta=2${\rm)}. Then there exists a neighborhood $N(\bar{\boldsymbol{\omega}})$ of \hspace{0.15em}$\bar{\boldsymbol{\omega}}$ such that, for any $\boldsymbol{\omega}_0\in N(\bar{\boldsymbol{\omega}})$, the sequence $\{\boldsymbol{\omega}_k\}$ generated by the {\rm LPA} with initial point $\boldsymbol{\omega}_0$ converges at a rate of $\frac{2}{\beta}$ to a solution $\boldsymbol{\omega}^{*}$ satisfying $G(\boldsymbol{\omega}^{*})\in C_{\hspace{0em}h}$. \\[-12pt]
\end{lemma}

\begin{lemma} \label{lem_glpa}
{\bf\rm(\cite{Hu2016}, Theorem 18)}
Let $\{\boldsymbol{\omega}_k\}$ be a sequence generated by the {\rm GLPA} and assume that $\{\boldsymbol{\omega}_k\}$ has a cluster point \hspace{0.1em}$\boldsymbol{\omega}^{*}$. Suppose that $\beta\in[1,2)$ and that $C_{h}$ be the set of weak sharp minima of order $\beta$ for $h$ near $G(\boldsymbol{\omega}^{*})$. Suppose further that $G$ is continuously differentiable with a Lipschitz continuous gradient $G^{\prime}$ near $\boldsymbol{\omega}^{*}$, and that \hspace{0.1em}$\boldsymbol{\omega}^{*}$ is a regular point of the inclusion. Then $G(\boldsymbol{\omega}^{*})\in C_{h}$, and $\{\boldsymbol{\omega}_k\}$ converges to $\boldsymbol{\omega}^{*}$ at a rate of $\frac{2}{\beta}$.
\end{lemma}

Note that $\beta\in[1,2)$ in Lemma \ref{lem_glpa} is lightly different from $\beta\in[1,2]$ in Lemma \ref{lem_lpa}, but both of them can find a globally optimal solution to optimization (\ref{cco}) since that $G(\boldsymbol{\omega}^{*})\in C_{h}$, equivalently, $h(G(\boldsymbol{\omega}^{*}))=h_{\text{min}}={(h\circ G)}_{\text{min}}$. \\[-5pt]

{\bf 4.2 Convergence Analysis for Sigmoid Networks.}
Let $B(\boldsymbol{z},r)$ denote an open ball of radius $r$ centered at $\boldsymbol{z}$, then we establish the local convergence of Algorithm \ref{lpa} by virtue of Lemma \ref{lem_lpa}.

\begin{theorem} \label{local} {\rm\bf (Local Convergence)}. Let $\beta\in[1,2]$ and $r>0$. Let $\{\boldsymbol{\theta}_k\}$ be a sequence generated by {\rm Algorithm} {\rm\ref{lpa}}, and \hspace{0.15em}$\bar{\boldsymbol{\theta}}\in\mathbb{R}^n$ be such that $F(\bar{\boldsymbol{\theta}})\in C_{\hspace{0.1em}\mathbb{L}}$ and $C_{\hspace{0.1em}\mathbb{L}}$ is the set of weak sharp minima of order $\beta$ for $\mathbb{L}$ on $B(F(\bar{\boldsymbol{\theta}}),r)$. If \hspace{0.15em}$\bar{\boldsymbol{\theta}}$ is a regular point of the inclusion, then there exist $t_0\geq 0$ and $\bar{r}>0$ such that for any $t>t_0$ and initial point \hspace{0.1em}$\boldsymbol{\theta}_0\in B(\bar{\boldsymbol{\theta}},\bar{r})$, the sequence $\{\boldsymbol{\theta}_k\}$ converges at a rate of $\frac{2}{\beta}$ to a globally optimal solution $\boldsymbol{\theta}^{*}$ and $F(\boldsymbol{\theta}^{*})\in C_{\hspace{0.1em}\mathbb{L}}$.
\end{theorem}

\begin{proof}
According to the assumptions of Lemma \ref{lem_lpa}, we need to verify the following four conditions. \\[-20pt]
\begin{enumerate}
    \item[(\romannumeral1)] \emph{Quasi-regularity condition}. By Proposition 3.3 in (\cite{Burke1995}), we know that any regular point of the inclusion $F(\boldsymbol{\theta})\in C_{\hspace{0.1em}\mathbb{L}}$ is also a quasi-regular point. Since $\bar{\boldsymbol{\theta}}$ is a regular point, $\bar{\boldsymbol{\theta}}$ is also a quasi-regular point of the inclusion $F(\boldsymbol{\theta})\in C_{\hspace{0.1em}\mathbb{L}}$, namely there exist $\delta>0$ and $r_0>0$ such that
    \begin{equation*}
    \Pi(\boldsymbol{\theta})\neq\emptyset ~~\text{and}~~ \text{\rm dist}(\boldsymbol{0},\Pi(\boldsymbol{\theta}))\leq \delta \text{\rm dist}(F(\boldsymbol{\theta}),C_{\hspace{0.1em}\mathbb{L}})~\text{for each}~ \boldsymbol{\theta}\in B(\bar{\boldsymbol{\theta}}, r_0),
    \end{equation*}
    where $\Pi(\boldsymbol{\theta}):=\{\Delta\boldsymbol{\theta}\in\mathbb{R}^n: F(\boldsymbol{\theta})+F^{\prime}(\boldsymbol{\theta})\Delta\boldsymbol{\theta}\in C_{\hspace{0.1em}\mathbb{L}}\}$ is the solution set of the linearized inclusion $F(\boldsymbol{\theta})+F^{\prime}(\boldsymbol{\theta})\Delta\boldsymbol{\theta}\in C_{\hspace{0.1em}\mathbb{L}}$. \\[-20pt]
    \item[(\romannumeral2)] \emph{Weak sharp minima}. In particular, we set $r_0\in(0,r)$. Naturally, $C_{\hspace{0.1em}\mathbb{L}}$ is the set of local weak sharp minima of order $\beta$ for $\mathbb{L}$ on $B(F(\bar{\boldsymbol{\theta}}),r_0)$ with constant $\alpha$ for some $\alpha>0$, due to the assumption and definition of the weak sharp minima. \\[-20pt]
    \item[(\romannumeral3)] \emph{Lipschitz continuous gradient}. Note that a differentiable function with a Lipschitz continuous gradient is second-order differentiable almost everywhere on $\Omega$. If $G$ is a second-order differentiable function, by the differential mean value theorem, it is obvious that the $K$-smoothness of $G$ is equivalent to the boundedness of $G^{\prime\prime}$, that is, $\Vert G^{\prime\prime}(\boldsymbol{\omega})\Vert\leq K$ for each $\boldsymbol{\omega}\in\Omega$. On the other hand, since $F$ defined by (\ref{composite}) is smooth on $\mathbb{R}^n$, $F^{\prime\prime}$ is continuous on $\mathbb{R}^n$. Naturally, $F^{\prime\prime}$ is bounded on the bounded subset $B(\bar{\boldsymbol{\theta}},r_0)$. Therefore, $F$ is continuously differentiable with a Lipschitz continuous gradient $F^{\prime}$ on $B(\bar{\boldsymbol{\theta}},r_0)$.\\[-20pt]
    \item[(\romannumeral4)] \emph{Large stepsize}. If $\beta=2$, we set $t_0=\frac{2\delta^2}{\alpha}$; otherwise, set $t_0=0$. 
\end{enumerate}

Hence, Lemma \ref{lem_lpa} is applicable and the conclusion follows.
\end{proof}

Furthermore, we analyze the convergence properties of Algorithm \ref{lpa} for the three common sigmoid networks. 

\begin{corollary}\label{sufficient} Let $\{\boldsymbol{\theta}_k\}$ be a sequence generated by {\rm Algorithm} {\rm\ref{lpa}}, and  \hspace{0.15em}$\bar{\boldsymbol{\theta}}\in\mathbb{R}^n$ be such that $F(\bar{\boldsymbol{\theta}})\in C_{\hspace{0.1em}\mathbb{L}}$.
If $F^{\prime}(\bar{\boldsymbol{\theta}})$ has full row rank, then there exists an $\bar{r}>0$ such that for any initial point \hspace{0.1em}$\boldsymbol{\theta}_0\in B(\bar{\boldsymbol{\theta}},\bar{r})$, we have\\[-20pt]
\begin{enumerate}
    \item[\rm(\romannumeral1)] for the sigmoid networks with the quadratic loss function, the sequence $\{\boldsymbol{\theta}_k\}$
    linearly converges to a globally optimal solution $\boldsymbol{\theta}^{*}$ and $F(\boldsymbol{\theta}^{*})=\boldsymbol{0}$, if $t$ is sufficiently large. \\[-20pt]
    \item[\rm(\romannumeral2)] for the sigmoid networks with the absolute loss function, the sequence $\{\boldsymbol{\theta}_k\}$ quadratically converges to a globally optimal solution $\boldsymbol{\theta}^{*}$ and $F(\boldsymbol{\theta}^{*})=\boldsymbol{0}$.\\[-20pt]
    \item[\rm(\romannumeral3)] for the sigmoid networks with the hinge loss function, the sequence $\{\boldsymbol{\theta}_k\}$ quadratically converges to a globally optimal solution $\boldsymbol{\theta}^{*}$ and $F(\boldsymbol{\theta}^{*})\geq\boldsymbol{1}$.
\end{enumerate}
\end{corollary}

\begin{proof}
According to the assumptions of Theorem \ref{local}, we need to verify the following two conditions.\\[-18pt]
\begin{enumerate}
    \item[(a)] \emph{Regularity condition}. Since the system of linear equations $W\boldsymbol{y}=\boldsymbol{0}$ has only zero solution if and only if the matrix $W$ has full column rank, $F^{\prime}(\bar{\boldsymbol{\theta}})$ with full row rank is equivalent to $\text{ker}(F^{\prime}(\bar{\boldsymbol{\theta}})^T)=\{\boldsymbol{0}\}$. Then, it follows that
    $$\text{ker}(F^{\prime}(\bar{\boldsymbol{\theta}})^T) \cap (C_{\mathbb{L}}-F(\bar{\boldsymbol{\theta}}))^{\ominus} = \{\boldsymbol{0}\}.$$
    Therefore, the regularity condition is satisfied.\\[-20pt]
    \item[(b)] \emph{Weak sharp minima}. Note that $\mathbb{L}_{\rm min}=0$; $C_{\hspace{0.1em}\mathbb{L}}=\{\boldsymbol{0}\}$ for the quadratic or absolute loss functions, and $C_{\hspace{0.1em}\mathbb{L}}\geq \boldsymbol{1}$ for the hinge loss function.\\
\end{enumerate} \ \\[-46pt]

    \hspace{-0.25em}(\romannumeral1)
    \begin{minipage}[t]{0.95\linewidth}
    In the case when $\mathbb{L}(\boldsymbol{z}) = \frac{1}{m} \Vert \boldsymbol{z}\Vert^2$, $\mathbb{L}(\boldsymbol{z})=\mathbb{L}_{\rm min}+\frac{1}{m}\text{dist}^2(\boldsymbol{z},C_{\hspace{0.1em}\mathbb{L}})$ for each $\boldsymbol{z}\in\mathbb{R}^m$. By the definition of weak sharp minima, we know that $C_{\hspace{0.1em}\mathbb{L}}=\{\boldsymbol{0}\}$ is the set of weak sharp minima of order $2$ for $\mathbb{L}$ on $\mathbb{R}^m$ with modulus $\frac{1}{m}$.
    \end{minipage}

    \hspace{-0.25em}(\romannumeral2)
    \begin{minipage}[t]{0.95\linewidth}
    In the case when $\mathbb{L}(\boldsymbol{z}) = \frac{1}{m} \Vert \boldsymbol{z} \Vert_1$, $\mathbb{L}(\boldsymbol{z}) \geq \frac{1}{m}\Vert \boldsymbol{z}\Vert=\mathbb{L}_{\rm min}+\frac{1}{m}\text{dist}(\boldsymbol{z},C_{\hspace{0.1em}\mathbb{L}})$ for each $\boldsymbol{z}\in\mathbb{R}^m$. In the same way, it shows that $C_{\hspace{0.1em}\mathbb{L}}=\{\boldsymbol{0}\}$ is the set of weak sharp minima of order $1$ for $\mathbb{L}$ on $\mathbb{R}^m$ with modulus $\frac{1}{m}$.
    \end{minipage}

    \hspace{-0.45em}(\romannumeral3)
     \begin{minipage}[t]{0.95\linewidth}
     In the case when $\mathbb{L}(\boldsymbol{z})=\frac{1}{m}\Vert(\boldsymbol{1}-\boldsymbol{z})_{+}\Vert_1$, $\mathbb{L}(\boldsymbol{z})\geq\frac{1}{m}\Vert(\boldsymbol{1}-\boldsymbol{z})_{+}\Vert= \mathbb{L}_{\rm min}+\frac{1}{m}\text{dist}(\boldsymbol{z},C_{\hspace{0.1em}\mathbb{L}})$ for each $\boldsymbol{z}\in\mathbb{R}^m$, which implies that $C_{\hspace{0.1em}\mathbb{L}}\geq \boldsymbol{1}$ is the set of weak sharp minima of order $1$ for $\mathbb{L}$ on $\mathbb{R}^m$ with modulus $\frac{1}{m}$. Therefore, the local weak sharp minima is satisfied for the three common sigmoid networks.
     \end{minipage}\\[0pt]

Hence, Theorem {\rm\ref{local}} is applicable and the conclusion follows.
\end{proof}


As we have seen, the weak sharp minima is often satisfied for sigmoid networks, and its order determines the convergence rate of the algorithm. To our surprise, a first-order algorithm even has a second-order convergence rate. In the following paragraphs, we establish the global convergence of Algorithm \ref{glpa} by virtue of Lemma \ref{lem_glpa}.

\begin{theorem} \label{global} {\rm\bf (Global Convergence)}. Let $\beta\in[1,2)$ and $r>0$. Let $\{\boldsymbol{\theta}_k\}$ be a sequence generated by {\rm Algorithm} {\rm\ref{glpa}}, and $\{\boldsymbol{\theta}_k\}$ have a cluster point \hspace{0.1em}$\boldsymbol{\theta}^{*}$ such that $C_{\hspace{0.1em}\mathbb{L}}$ be the set of weak sharp minima of order $\beta$ for $\mathbb{L}$ on $B(F(\boldsymbol{\theta}^{*}),r)$. If \hspace{0.1em}$\boldsymbol{\theta}^{*}$ is a regular point of the inclusion, then $\{\boldsymbol{\theta}_k\}$ converges at a rate of $\frac{2}{\beta}$ to a globally optimal solution $\boldsymbol{\theta}^{*}$ and $F(\boldsymbol{\theta}^{*})\in C_{\hspace{0.1em}\mathbb{L}}$.
\end{theorem}

\begin{proof}
According to the assumptions of Lemma \ref{lem_glpa}, we need to verify the following three conditions. \\[-18pt]
\begin{enumerate}
    \item[(\romannumeral1)] \emph{Regularity condition}. Since the cluster point $\boldsymbol{\theta}^{*}$ is a regular point of the inclusion $F(\boldsymbol{\theta})\in C_{\hspace{0.1em}\mathbb{L}}$, the regularity condition is satisfied.\\[-20pt]
    \item[(\romannumeral2)] \emph{Weak sharp minima}. Since $C_{\hspace{0.1em}\mathbb{L}}$ is the set of weak sharp minima of order $\beta$ for $\mathbb{L}$ on $B(F(\boldsymbol{\theta}^{*}),r)$ for some $r>0$ and $\beta\in[1,2)$, the local weak sharp minima is satisfied. \\[-20pt]
    \item[(\romannumeral3)] \emph{Lipschitz continuous gradient}. By (\romannumeral3) in the proof of Theorem \ref{local}, we know that $F$ is continuously differentiable with a Lipschitz continuous gradient $F^{\prime}$ on $B(\boldsymbol{\theta}^{*},r)$.\\[-20pt]
\end{enumerate}

Hence, Lemma \ref{lem_glpa} is applicable and the conclusion follows.
\end{proof}

We can see that Algorithm \ref{glpa} has the same conclusion and convergence rate as Algorithm \ref{lpa} under the same assumptions. Next, we show the global convergence of Algorithm \ref{glpa} for two non-convex and non-smooth sigmoid networks.

\begin{corollary} \label{suffg}
Let $\{\boldsymbol{\theta}_k\}$ be a sequence generated by {\rm Algorithm} {\rm\ref{glpa}} for the sigmoid networks with absolute or hinge loss functions, and $\{\boldsymbol{\theta}_k\}$ have a cluster point \hspace{0.1em}$\boldsymbol{\theta}^{*}$. If $F^{\prime}(\boldsymbol{\theta}^{*})$ has full row rank, then $\{\boldsymbol{\theta}_k\}$ quadratically converges to a globally optimal solution $\boldsymbol{\theta}^{*}$ and $F(\boldsymbol{\theta}^{*})\in C_{\hspace{0.1em}\mathbb{L}}$. 
\end{corollary}

\begin{proof}
According to the assumptions of Theorem \ref{global}, we need to verify the following two conditions. \\[-18pt]
\begin{enumerate}
    \item[(a)] \emph{Regularity condition}. By (a) in the proof of Corollary \ref{sufficient},  the full row rank of $F^{\prime}(\boldsymbol{\theta}^{*})$ implies that the cluster point $\boldsymbol{\theta}^{*}$ is a regular point of the inclusion. Therefore, the regularity condition is satisfied.\\[-20pt]
    \item[(b)] \emph{Weak sharp minima}. By (b) in the proof of Corollary \ref{sufficient}, we know that $C_{\hspace{0.1em}\mathbb{L}}$ is the set of weak sharp minima of order 1 for $\mathbb{L}$ on $\mathbb{R}^m$ with modulus $\frac{1}{m}$. Therefore, the local weak sharp minima is satisfied for the two sigmoid networks. \\[-20pt]
\end{enumerate}

Hence, Theorem \ref{global} is applicable and the conclusion follows.
\end{proof}

Note that $F^{\prime}(\bar{\boldsymbol{\theta}})$ with full row rank, namely $\text{\rm rank}(F^{\prime}(\bar{\boldsymbol{\theta}}))=m$, where $m$ is the amount of training data, is the sufficient condition of the regularity condition; and it is also the necessary condition when \hspace{0.05em}$C_{\hspace{0.1em}\mathbb{L}}$ is a singleton set and $F^{\prime}(\bar{\boldsymbol{\theta}})\in C_{\hspace{0.1em}\mathbb{L}}$. Hence the convergence results can be directly related to the amount of training data. Next, we show the following convergence property of the LPA-type algorithms in a finite number of iterations.



\begin{theorem} \label{suffr} {\rm \bf (Sufficient Condition)}.
If the {\rm LPA-type} algorithm stops at the $k$th iteration with $\text{\rm rank}(G^{\prime}(\boldsymbol{\omega}_k))=l$, then $\boldsymbol{\omega}_k$ is a globally optimal solution to the convex composite optimization {\rm(\ref{cco})} and $G(\boldsymbol{\omega}_k)\in C_{h}$.
\end{theorem}

\begin{proof}
     Since the subproblem of the LPA-type algorithms is an unconstrained convex optimization problem, its necessary and sufficient optimality conditions imply that
    \begin{equation*}
        \boldsymbol{0}\in G^{\prime}(\boldsymbol{\omega}_k)^T \partial h(G(\boldsymbol{\omega}_k)+G^{\prime}(\boldsymbol{\omega}_k)\Delta\boldsymbol{\omega}_k)+\frac{1}{t}\Delta\boldsymbol{\omega}_k ~\text{for each}~ k,
    \end{equation*}
    where $\partial h(\boldsymbol{z})$ is the subdifferential of the convex function $h(\boldsymbol{z})$. The stopping criterion $\Delta\boldsymbol{\omega}_k=\boldsymbol{0}$ of the algorithms shows that
    \begin{equation*}
        \boldsymbol{0}\in  G^{\prime}(\boldsymbol{\omega}_k)^T \partial h(G(\boldsymbol{\omega}_k)).
    \end{equation*}
    By $\text{rank}(G^{\prime}(\boldsymbol{\omega}_k))=l$, equivalently, the full column rank of $G^{\prime}(\boldsymbol{\omega}_k)^T$, it follows that
    \begin{equation*}
        \boldsymbol{0}\in\partial h(G(\boldsymbol{\omega}_k)).
    \end{equation*}
    By the necessary and sufficient optimality conditions of the convex optimization, it shows that $G(\boldsymbol{\omega}_k)$ is a globally optimal solution to $h$, equivalently, $G(\boldsymbol{\omega}_k)\in C_{h}$. Hence the proof is complete. 
\end{proof} 

  Theorem \ref{suffr} also shows that $\text{rank}(F^{\prime}(\boldsymbol{\theta}_k))=m$ is the first-order sufficient condition of sigmoid networks when the LPA-type algorithm stops at the $k$th iteration. It is no surprise that there is a unified conclusion on the non-convex and possibly non-smooth sigmoid networks, thanks to the unified composite optimization framework and the convex subproblem.



 We have seen that the full row rank is a critical condition for the convergence analysis of sigmoid networks. This condition is of great theoretical and applied significance, especially since it can provide a general guide for setting the network size. In order to guarantee the reliability of the algorithm, we can ensure that $F^{\prime}(\bar{\boldsymbol{\theta}})\in\mathbb{R}^{m\times n}$ is of full row rank, which implies that $n=(d+2)q+1\geq m$, where $d$ is the dimension of the input, and $q$ is the number of hidden neurons. So we have the following corollary.

 \begin{corollary} \label{net_size}
 If {\rm$\hspace{0.1em}\text{rank}(F^{\prime}(\bar{\boldsymbol{\theta}}))=m$},
 then we have a lower bound on the network size given by
\begin{equation}\label{size}
    q\geq\left\lceil\frac{m-1}{d+2}\right\rceil.
\end{equation}
 \end{corollary}

Clearly, the lower bound on the network size is directly proportional to the amount of training data and inversely proportional to the dimension of the input. That is, the lower bound on the network size is adapted to the problem size, so we call this lower bound the ``adaptive network size". Moreover, each row of the Jacobian matrix $F^{\prime}(\boldsymbol{\theta})$ is the gradient of the fitting function $f(\boldsymbol{x};\boldsymbol{\theta})$ at the corresponding data point. In a general sense, as the number of hidden neurons increases, the information contained in the gradient increases. As a result, the rank of the Jacobian matrix will also increase or be equal to $m$. Thus, the full row rank of $F^{\prime}(\bar{\boldsymbol{\theta}})$ can be satisfied in a theoretical sense by choosing the network size sufficiently large. In conclusion, the LPA-type algorithms are almost always reliable. 


\section{Numerical Experiment}
\label{part5}

Sigmoid networks are often used to solve regression and classification tasks, so we shall use our algorithms for both tasks. We train the sigmoid networks on the training dataset and demonstrate the performance on the test dataset. Note that we will use the adaptive network size, namely the lower bound on the network size given by Corollary \ref{net_size}, to build the sigmoid networks, which is sufficient to solve problems effectively.

{\bf 5.1 Regression on Scattered Data}. Franke's function is a standard test function for 2D scattered data fitting of the form
\begin{align*}
    g(x_1,x_2) = & \frac{3}{4}e^{-1/4((9x_1-2)^2+(9x_2-2)^2)} +\frac{3}{4}e^{-(1/49)(9x_1+1)^2-(1/10)(9x_2+1)^2}\\
    & +\frac{1}{2}e^{-1/4((9x_1-7)^2+(9x_2-3)^2)} -\frac{1}{5}e^{-(9x_1-4)^2-(9x_2-7)^2},
\end{align*} 
and its graph in the unit square in $\mathbb{R}^2$ is shown on the left of Figure \ref{Haltonp}. One can see that Franke's function is a complex function with two Gaussian peaks and a small trough. We generate 289 training data points and 121 test data points using the Halton sequence. The points are uniformly distributed in the unit square in $\mathbb{R}^2$, and the result is shown on the right of Figure \ref{Haltonp}.

Considering the observational errors, we also add small white Gaussian noise to the training data to reflect the real case, that is, $y_i=g(x^1_i,x^2_i) + |\xi_i|, ~\text{and}~ \xi_i\sim N(0,\tilde{\sigma}^2)$, where $N(0,\tilde{\sigma}^2)$ is a Gaussian distribution with a mean of $0$ and a standard deviation of $\tilde{\sigma}$. All numerical experiments are implemented in Python 3.9. We generate the positive Gaussian noise using $\frac{1}{\sqrt{2\pi}\tilde{\sigma}}\cdot\text{uniform}(0,1)$.
The performance measure we choose for the regression task is the root mean squared error (RMS-error):
	\begin{equation*}
		\text{RMS-error} = \frac{1}{\sqrt{M}} \left( \sum_{i=1}^{M} (\tilde{y}_i-y_i)^2 \right)^{\frac{1}{2}},
	\end{equation*}
 where $\tilde{y}_i$ is the predicted value and $y_i$ is the actual value.

\begin{figure}[H]
\flushleft
\subfigure{
\hspace{0.5em}\includegraphics[width=0.55\linewidth]{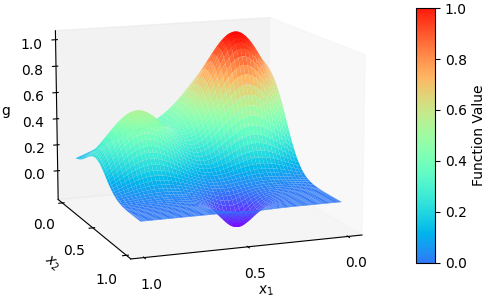}
}
\subfigure{
\hspace{1.5em}\includegraphics[width=0.33\linewidth]{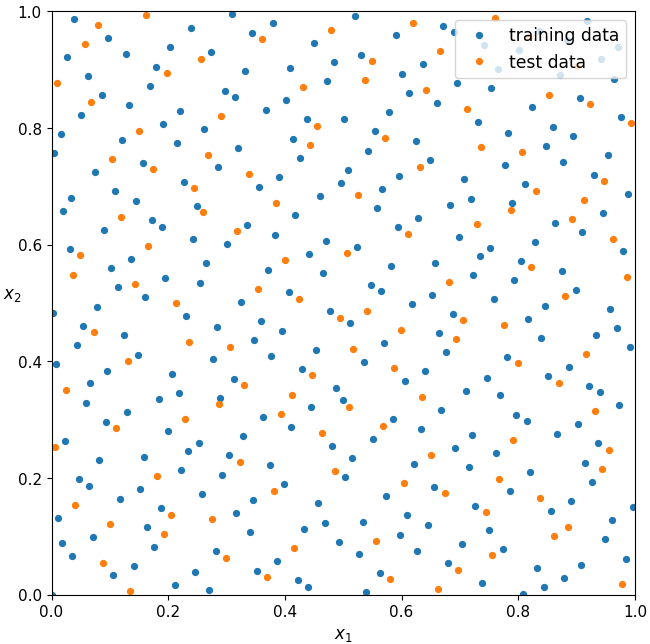}}
\caption{The graph of Franke's test function (left) and a set of 289 training data points and 121 test data points in the unit square in $\mathbb{R}^2$ (right).}
\label{Haltonp}
\vspace{-1.5em}
\end{figure}

When implementing the LPA-type algorithms for the sigmoid networks with a {\bf quadratic loss function}, we set $\tilde{\sigma}=100$, $\boldsymbol{\theta}_0=\boldsymbol{0}$, and the stopping criterion as $\Vert\Delta\boldsymbol{\theta}_k\Vert<$1e-2. For the inequality (\ref{rule}) in Algorithm \ref{glpa},
 we set $\tau=0.5$, $c=1$e-3, and the maximum number of iterations for the backtracking line-search as $10$ (indeed, one iteration is enough in most cases, that is, $\eta_k=1$ is often used). According to (\ref{size}),  we can set $q\geq 72$ to guarantee the reliability of the algorithms. For the case when $q=72$ and $t=$1e5, the performance of the algorithms is shown in Table \ref{tab_qua} and Figure \ref{L2_loss}. \\[-18pt]
 \begin{table}[H]
        \vspace{-5pt}
        \caption{The performance of regression on Franke's function (using quadratic loss).}
        \label{tab_qua}
        \vspace{5pt}
        \centering
        \begin{tabular}{ccccc}
        \hline
         & & & & \\[-12pt]
        & \multicolumn{2}{c}{LPA} & \multicolumn{2}{c}{GLPA} \\[2pt]
        & RMS-error & Max-error & RMS-error & Max-error \\[2pt]
        \hline \\[-10pt]
        No noise & 2.9525e-3 & 1.4736e-2 & 2.7790e-3 & 1.1547e-2 \\[2pt]
        Gaussian noise & 3.4364e-3 & 1.2678e-2 & 3.7613e-3 & 1.5765e-2 \\[2pt]
        \hline
        \end{tabular}
        \vspace{-1.2em}
    \end{table}
\begin{figure}[H]
\centering
\subfigure{
\includegraphics[width=0.45\linewidth]{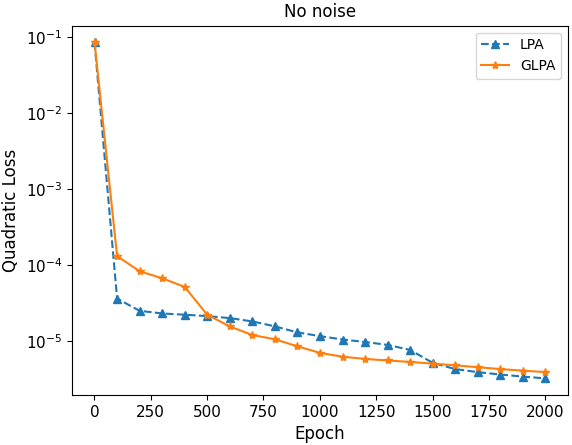}
}
\subfigure{
\includegraphics[width=0.45\linewidth]{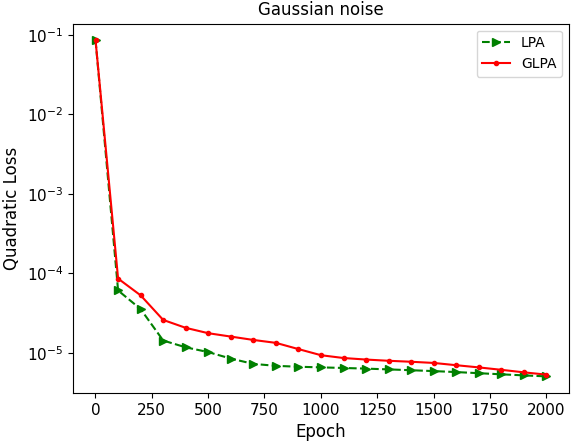}}
\caption{The variation of the quadratic loss during training. The training loss of the four experiments: (\romannumeral1) No noise: 3.1935e-6 and 5.0306e-6; (\romannumeral2)
Gaussian noise: 3.8656e-6 and 5.2940e-6.}
\label{L2_loss}
\label{influence}
\vspace{-1.2em}
\end{figure} 

    As we can see, the LPA-type algorithms solve the regression tasks well, and they are robust even when the data is perturbed by the noise with a mean of 2.0094e-3 and a maximum of 3.9894e-3. The results show that the training loss is less than 5.2940e-6 for all test cases. In other words, our algorithms can obtain an ideal solution for this task. We find that the monotonic decrease of the objective function occurs at almost every iteration of the LPA. It is almost a descent algorithm. Through multiple experiments, we also find that the performance of the LPA depends on the choice of the initial point, but the GLPA is not affected by this. Thus, we conjecture that the GLPA for sigmoid networks with the quadratic loss function can converge globally under certain conditions. This will be explored in our future work.


    Indeed, the LPA-type algorithms using small-scale networks can solve the problem as well. The illustration is shown on the left of Figure \ref{influence}. Moreover, the performance of the algorithms is also affected by the stepsize of the subproblem. This is shown on the right of Figure \ref{influence}.

\begin{figure}[H]
\flushleft
\subfigure{
\hspace{0em}\includegraphics[width=0.45\linewidth]{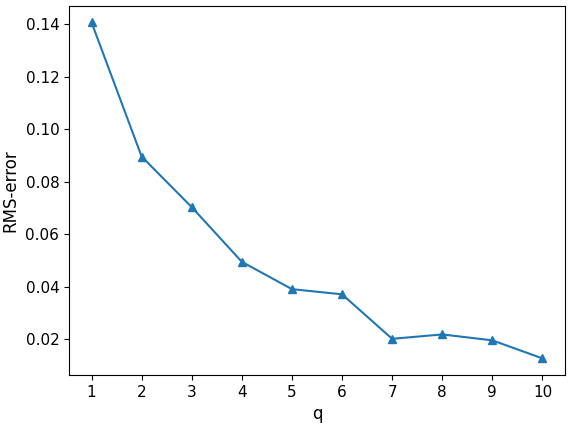}
}
\subfigure{
\hspace{1.5em}\includegraphics[width=0.45\linewidth]{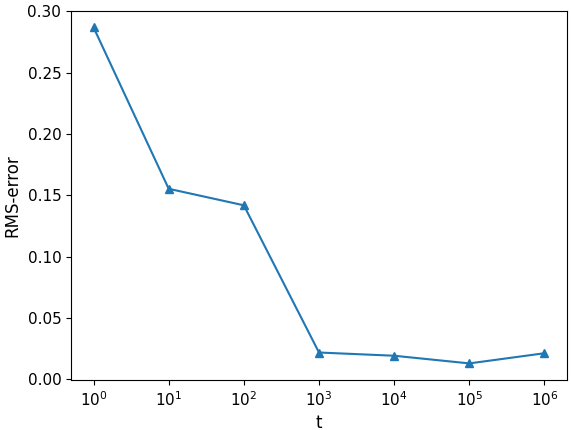}} \\[-7pt]
\caption{Execute Algorithm \ref{glpa} by varying the number $q$ of hidden neurons when $t=$1e5 (left) and the stepsize $t$ of the subproblem when $q=10$ (right).} 
\label{influence}
\end{figure} \ \\[-45pt]  

Corollary \ref{suffg} shows that Algorithm \ref{glpa} using absolute or hinge loss functions can converge globally. For simplicity, the rest of this section is devoted to demonstrating the performance of Algorithm \ref{glpa}. When implementing the GLPA for the sigmoid networks with an {\bf absolute loss function}, we still use the same parameter values as in the previous experiments. For Algorithm \ref{admm}, we set $\epsilon=\rho=$1e-2, $\Delta\boldsymbol{\theta}^0=\boldsymbol{0}$, $\boldsymbol{\lambda}^0=\boldsymbol{0}$, and the maximum number of ADMM iterations as 20. For the case when $q=72$ and $t=$1e5, the performance of the algorithm is shown
in Table \ref{tab2} and Figure \ref{L1_loss}.\\[-18pt]

\begin{table}[H]
        \vspace{-5pt}
        \caption{The performance of regression on Franke's function (using absolute loss).}
        \label{tab2}
        \vspace{5pt}
        \centering
        \hspace{0.5em}\begin{tabular}{ccc}
        \hline
        & & \\[-12pt]
         & \multicolumn{2}{c}{GLPA} \\[2pt]
         & RMS-error & max-error \\[2pt]
        \hline \\[-10pt]
        No noise & 2.2093e-4 & 8.4516e-4 \\[2pt]
        Gaussian noise & 8.4138e-4 & 4.3988e-3 \\[2pt]
        \hline
        \end{tabular}
    \end{table} \ \\[-39pt]

    \begin{figure}[H]
        \centering
        \hspace{-1em}\includegraphics[width=0.57\linewidth]{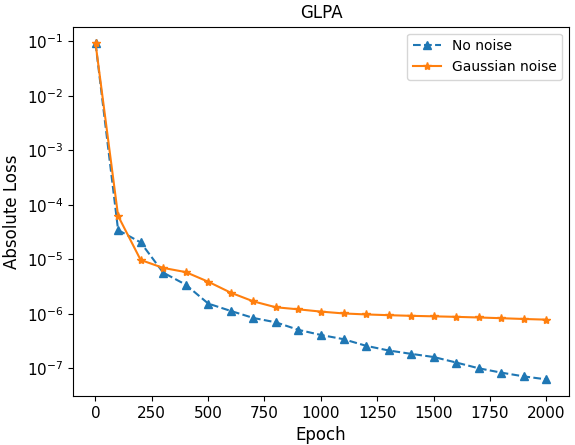} \\[-7pt]
        \caption{The variation of the absolute loss during training. The training loss in both experiments is 6.2393e-8 and 7.7930e-7.}
        \label{L1_loss}
        \vspace{-1.5em}
    \end{figure}

    The training loss in both experiments is less than 7.7930e-7, which shows that the GLPA obtains a better solution for sigmoid networks. Obviously, this result is more in line with the actual needs of regression tasks.

{\bf 5.2 Classification on Handwritten Digits}. The digits dataset from scikit-learn contains 1797 samples, each with 64 elements corresponding to an image of 8$\times$8 pixels, and with target attribute 0, 1, $\dots$, 9. Some of the samples are shown in Figure \ref{digits_30}. \\[-25pt]

    \begin{figure}[H]
        \centering
    \includegraphics[width=0.66\linewidth]{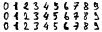}
        \caption{The first 30 samples of the digits dataset from scikit-learn.}
        \label{digits_30} \ \\[-40pt]
    \end{figure}

    We create four binary classification tasks, each to classify two digits: 0 and 1; 2 and 5; 3 and 7; 6 and 9. For each task, we take 70\% of the selected samples as the training data and the rest as the test data. Here we run four algorithms on these tasks, including the GLPA and three other popular and practical tools in the machine learning community, SGDM, RMSProp and Adam. We also use the same parameter settings for the GLPA as the previous experiments. The only difference is that we set $q=4$ by Corollary \ref{net_size} and the number of ADMM iterations does not exceed 10. For the other algorithms, implemented with PyTorch, we set the learning rate as 1e-3, the momentum as 0.9, and the number of iterations as 1000. For the case when $q=4$, the running results of the four algorithms are shown in Table \ref{hinge_e} and Figure \ref{hinge_loss}.

    Three observations are indicated by the running results: (\romannumeral1) The small training loss shows that the GLPA can obtain excellent solutions to classification problems, and the training loss of the GLPA is generally smaller than the other algorithms. (\romannumeral2) The GLPA has a much smaller number of iterations, thanks to its quadratic convergence rate in this case. It is striking that a first-order algorithm (GLPA) even has a second-order convergence rate. (\romannumeral3) The adaptive network size given by Corollary \ref{net_size} is sufficient to construct an ideal sigmoid network that solves the problem effectively. Hence Corollary \ref{net_size} does provide a good guide for setting the size of sigmoid networks. \\[-30pt]

 \begin{table}[H]
        \vspace{-5pt}
        \caption{The performance of classification on handwritten digit (using hinge loss).}
        \label{hinge_e}
        \vspace{5pt}
        \centering
        \begin{tabular}{ccccc}
        \hline
         & & & & \\[-12pt]
       Classified & \multicolumn{2}{c}{GLPA} & \multicolumn{2}{c}{SGDM (RMSProp, Adam)} \\[2pt]
       Digits & Training errors & Test errors & Training errors & Test errors \\[2pt]
        \hline \\[-10pt]
        0 - 1 & 0 / 252 & 0 / 108 &  0 / 252 &  0 / 108 \\[2pt]
        2 - 5 & 0 / 251 & 0 / 108 &  0 / 251 &  0 / 108  \\[2pt]
        3 - 7 & 0 / 253 & 0 / 109 &  0 / 253 &  0 / 109   \\[2pt]
        6 - 9 & 0 / 252 & 1 / 109 &  0 / 252 &  1 / 109   \\[2pt]
        \hline
        \end{tabular}
    \end{table}

\begin{figure}[H]
\flushleft
\subfigure{
\hspace{0em}\includegraphics[width=0.45\linewidth]{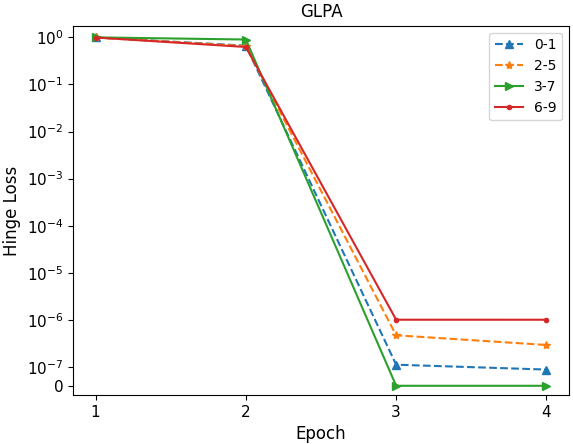}
}
\subfigure{
\hspace{1.5em}\includegraphics[width=0.45\linewidth]{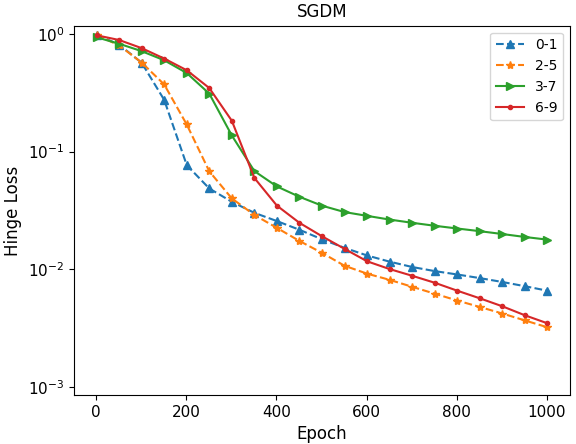}
} \vspace{0.1pt}
\subfigure{
\hspace{0em}\includegraphics[width=0.45\linewidth]{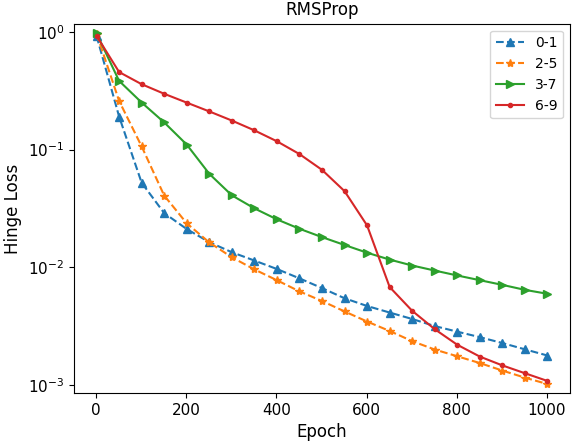}
}
\subfigure{
\hspace{1.5em}\includegraphics[width=0.45\linewidth]{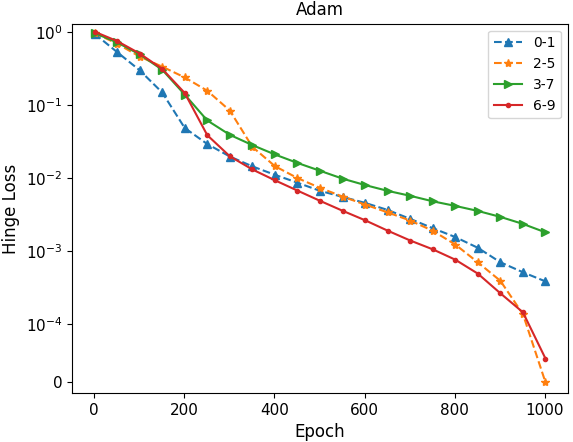}
}
\caption{The variation of the hinge loss during training. The training loss of the four binary classification tasks: (\romannumeral1) GLPA: 8.8007e-8, 2.9285e-7, 0.0 and 1.0065e-6; (\romannumeral2) SGDM: 6.5701e-3, 3.2118e-3, 1.7844e-2 and 3.4731e-3; (\romannumeral3) RMSProp: 1.7785e-3, 1.0178e-3, 5.9491e-3 and 1.0813e-3; (\romannumeral4) Adam: 3.8339e-4, 0.0, 1.8042e-3 and 3.3567e-5.} 
\label{hinge_loss}
\vspace{-1.2em}
\end{figure}

The essence of Corollary \ref{net_size} is to guarantee that the number of parameters in neural networks is not smaller than the amount of training data, and that a sufficient number of parameters ensure the feasibility of the networks. In our view, it is as if the information of a data point could be extracted by a single parameter in the model. Inspired by this, we think it can also serve as a general guide for setting the size of neural networks. It is well known that how to set the number of hidden neurons in neural networks is still an open problem, and it is usually adjusted by trial and error in practice. As stated above, we suggest that the number of hidden neurons can be specified by trial and error starting from the adaptive network size, which can avoid certain blindness at the beginning of the trial. This general rule deserves to be tried and further verified in practice.


\section{Future Work}
\label{part6}

Although we only show the composite optimization algorithms for the three-layer sigmoid networks, our algorithms are also applicable to the more complex sigmoid networks, such as the sigmoid networks with multiple hidden layers, with multiple outputs, and with output layer neurons that are processed with sigmoid functions. In the design of model (\ref{composite}), the convexity of the outer function $\mathbb{L}$ is due to the convex loss function $L$, and the smoothness of the inner function $F$ is due to the smooth fitting function $f$. So the algorithms can be used to solve the sigmoid networks whenever we maintain the convexity of $L$ and the smoothness of $f$ (note that $f$ is always smooth in sigmoid networks). It is not difficult to solve the general sigmoid networks with convex loss functions using our algorithms by setting the same form of $\mathbb{L}$ and $F$ as the case of one hidden layer. As a matter of fact, the composite structure (\ref{composite}) can provide a unified framework for the development and analysis of sigmoid networks, especially for the non-convex and non-smooth optimization problems. Moreover, the various composite structures in neural networks pose more challenges for the study of composite optimization algorithms. The breakthrough of composite optimization algorithms will also drive the development of neural network learning algorithms. Last but not least, the convergence results of convex composite optimization (\ref{cco}) in the literature all seem to be established on $G(\boldsymbol{\omega}^{*})\in C_{h}$. While the more general convergence theorems should be established possibly on $G(\boldsymbol{\omega}^{*})\notin C_{h}$, which is still an open problem in the area of composite optimization. In view of this, we will explore this issue further.

\subsection*{Acknowledgments}

The research was supported in part by the National Natural Science Foundation of China under grants 12071157 and 12026602, and the Natural Science Foundation of Guangdong 2020B1515310013. Qi Ye is the corresponding author.

\vskip 0.2in
\bibliography{ref}

\end{document}